\documentclass[a4paper,12pt]{article}
\usepackage[top=2.5cm,bottom=2.5cm,left=2.5cm,right=2.5cm]{geometry}
\usepackage{cite, amsmath, amssymb}
\usepackage[margin=1cm,%
                font=small,%
                format=hang,%
                labelsep=period,%
                labelfont=bf]{caption}
\usepackage{amssymb}
\usepackage{graphicx}
\newenvironment{proof}{{\noindent \it Proof.}}{\hfill $\blacksquare$\par}
\usepackage{latexsym}
\usepackage{graphicx,booktabs,multirow}
\usepackage{tikz}
\usetikzlibrary{decorations.pathreplacing}
\usetikzlibrary{intersections}
\usepackage{enumerate}
\usepackage{graphicx,booktabs,multirow}
\usepackage{appendix}
\newtheorem{theorem}{Theorem}[section]

\newtheorem{lemma}[theorem]{Lemma}

\newtheorem{problem}{Problem}[section]
\begin{document}

\title{Ordering Chemical Graphs by Sombor Indices and Its Applications}
\author{Hechao Liu$^{1}$, Lihua You$^{1,}$\thanks{Corresponding author}, Yufei Huang$^{2}$
 \\
{\small $^1$School of Mathematical Sciences, South China Normal University,}\\ {\small Guangzhou, 510631, P. R. China}\\
 \small {\tt hechaoliu@m.scnu.edu.cn},\quad  \small {\tt ylhua@scnu.edu.cn}\\
{\small $^2$Department of Mathematics Teaching, Guangzhou Civil Aviation College,}\\ {\small Guangzhou, 510403, P. R. China}\\
\small {\tt fayger@qq.com}\\\\
\small { (Received March 17, 2021)}
}
\date{}
\maketitle
\begin{abstract}
Topological indices are a class of numerical invariants that predict certain physical and chemical properties of molecules.
Recently, two novel topological indices, named as Sombor index and reduced Sombor index, were introduced by Gutman, defined as
$$SO(G)=\sum_{uv\in E(G)}\sqrt{d_{G}^{2}(u)+d_{G}^{2}(v)},$$
$$SO_{red}(G)=\sum_{uv\in E(G)}\sqrt{(d_{G}(u)-1)^{2}+(d_{G}(v)-1)^{2}},$$
where $d_{G}(u)$ denotes the degree of vertex $u$ in $G$.

In this paper, our aim is to order the chemical trees, chemical unicyclic graphs, chemical bicyclic graphs and chemical tricyclic graphs with respect to Sombor index and reduced Sombor index. We determine the first fourteen minimum chemical trees, the first four minimum chemical unicyclic graphs, the first three minimum chemical bicyclic graphs, the first seven minimum chemical tricyclic graphs. At last, we consider the applications of reduced Sombor index to octane isomers.
\end{abstract}
\restoregeometry
\baselineskip=0.30in
\maketitle

\makeatletter
\renewcommand\@makefnmark%
{\mbox{\textsuperscript{\normalfont\@thefnmark)}}}
\makeatother

\section{Introduction}

Let $G$ be a simple connected graph with vertex set $V(G)$ and edge set $E(G)$. Denote by $N_{G}(u)$ the set of vertices adjacent to $u$ in $G$ for every $u\in V(G)$. The degree $d_{G}(u)$ of $u$ in $G$ is the cardinality of $N_{G}(u)$. Let $\Delta(G)$ (or simply $\Delta$) be the maximum degree of $G$. Let $n_{i}(G)$ (or simply $n_{i}$) be the number of vertices with degree $i$ in $G$.
Denote by $m_{i,j}(G)$ the numbers of edges connected a vertex with degree $i$ and a vertex with degree $j$ in $G$.
In this paper, all notations and terminologies used but not defined can refer to the textbook \cite{jaus2008}.

The Sombor index ($SO(G)$ for short) and reduced Sombor index ($SO_{red}(G)$ for short) of a graph $G$ are defined as \cite{gumn2021}
$$SO(G)=\sum_{uv\in E(G)}\sqrt{d_{G}^{2}(u)+d_{G}^{2}(v)},$$
$$SO_{red}(G)=\sum_{uv\in E(G)}\sqrt{(d_{G}(u)-1)^{2}+(d_{G}(v)-1)^{2}}.$$
Shortly after, Deng et al. \cite{dengt2021} determined the maximum Sombor indices of chemical trees. Cruz et al. \cite{rirm2021} determined the extremal values of some chemical graphs. Also, Red\v{z}epovi\'{c} \cite{redz2021} studied chemical applicability of Sombor indices. Milovanovi\'{c} et al. \cite{milo2021} considered the bounds of Sombor indices and the relations between Sombor indices and other indices. Furthermore, Liu et al. \cite{tliu2021} obtained some bounds for reduced Sombor index of graphs with given several parameters and some special graphs, they also obtained the the expected values of reduced Sombor index in random polyphenyl chains, the bounds of reduced Sombor spectrum radius and energy. For more details of Sombor indices, we refer to \cite{aiva2021,dxsa2021,gumn2021,guma2021,hoxu2021,wmlf2021,trdo2021}.

The chemical graph is a graph with $d_{G}(u)\leq 4$ for all $u\in V(G)$.
Ordering chemical graphs by some topological indices is an interesting problem.
Ghalavand and Ashrafi had done a lot of work on ordering chemical graphs by some topological indices such as Wiener polarity index \cite{argh2017}, sum exdeg index \cite{laas2017}, total irregularity \cite{gvrf2020}, forgotten coindex \cite{ganf2020}, Randi\'{c} index and sum-connectivity index\cite{ghas2018} and hyper-Zagreb index \cite{ghor2019}. For more related papers can be find in \cite{alid2018,ggda2017,jcli2021} and references cited therein.

Motivated by \cite{ghas2018,ggda2017}, our aim is to consider the similar issues regarding (reduced) Sombor index.
In this paper, we determine the first fourteen minimum chemical trees, the first four minimum chemical unicyclic graphs, the first three minimum chemical bicyclic graphs, the first seven minimum chemical tricyclic graphs. At last, we consider the applications of reduced Sombor index to octane isomers.

\section{Preliminaries}

Here are some important transformations that will be used in the proof of main results.

\begin{lemma}\label{l-22}
Let $G_{0}$ be a connected graph with vertices $u_{1},u_{2},u_{3},u_{4}$ $($$d_{G_{0}}(u_{1})=1$, $d_{G_{0}}(u_{2})$ $=2$, $d_{G_{0}}(u_{3})=3$ or $4$, $d_{G_{0}}(u_{4})=1$, $\{u_{1}u_{2},u_{3}u_{4}\}\subseteq E(G_{0})$$)$. Suppose that $P=v_{1}v_{2}\cdots v_{l}$ is a path. Denote by $G_{1}$ the graph gotten from $G_{0}$, $P$ by attaching vertices $u_{1}v_{1}$.
Let $G_{2}=G_{1}-u_{1}v_{1}+u_{4}v_{1}$. Then $SO(G_{1})> SO(G_{2})$ and $SO_{red}(G_{1})> SO_{red}(G_{2})$.
\end{lemma}
\begin{proof}
By the definition of Sombor index, we have
\begin{align*}
SO(G_{2})-SO(G_{1})
=&\sqrt{1^{2}+2^{2}}+\sqrt{d_{G_{0}}^{2}(u_{3})+2^{2}}-[\sqrt{2^{2}+2^{2}}+\sqrt{d_{G_{0}}^{2}(u_{3})+1^{2}}]\\
=&\sqrt{5}+\sqrt{d_{G_{0}}^{2}(u_{3})+4}-2\sqrt{2}-\sqrt{d_{G_{0}}^{2}(u_{3})+1^{2}}.
\end{align*}

If $d_{G_{0}}^{2}(u_{3})=3$, then $SO(G_{2})-SO(G_{1})=\sqrt{5}+\sqrt{13}-2\sqrt{2}-\sqrt{10}<0$.

If $d_{G_{0}}^{2}(u_{3})=4$, then $SO(G_{2})-SO(G_{1})=\sqrt{5}+\sqrt{20}-2\sqrt{2}-\sqrt{17}<0$.

Therefore $SO(G_{1})> SO(G_{2})$.
In a similar way, we also have $SO_{red}(G_{1})> SO_{red}(G_{2})$.
This completes the proof.
\end{proof}


\begin{lemma}\label{l-23}
Let $G_{0}$, $G$ be two connected graphs with vertices $x\in V(G_{0})$,  $y\in V(G)$ $($ $d_{G_{0}}(x)=1$ or $2$ and $d_{G}(y)=2$ or $3$ $)$. Suppose that $P_{1}=u_{1}u_{2}\cdots u_{k}$ and $P_{2}=v_{1}v_{2}\cdots v_{l}$ are two paths . Denote by $G_{1}$ the graph gotten from $G_{0}$, $G$, $P_{1}$ and $P_{2}$ by attaching vertices $v_{1}x$, $u_{1}x$ and $u_{k}y$. Let $G_{2}=G_{1}-\{u_{1}x,u_{k}y\}+\{xy,u_{1}v_{l}\}$. Then $SO(G_{1})> SO(G_{2})$ and $SO_{red}(G_{1})> SO_{red}(G_{2})$.
\end{lemma}
\begin{proof}
We consider the following two cases.

\noindent {\bf Case 1}. $l=1$.
\begin{align*}
SO(G_{2})-SO(G_{1})
=&\sqrt{1^{2}+2^{2}}+\sqrt{(d_{G_{0}}(x)+2)^{2}+(d_{G_{0}}(y)+1)^{2}}\\
& -[\sqrt{(d_{G_{0}}(x)+2)^{2}+1^{2}}+\sqrt{(d_{G_{0}}(y)+1)^{2}+2^{2}}].
\end{align*}
Since $d_{G_{0}}(x)=1$ or $2$ and $d_{G_{0}}(y)=2$ or $3$.

If $d_{G_{0}}(x)=1$ and $d_{G_{0}}(y)=2$, then $SO(G_{2})-SO(G_{1})=\sqrt{5}+3\sqrt{2}-\sqrt{10}-\sqrt{13}<0$.

If $d_{G_{0}}(x)=1$ and $d_{G_{0}}(y)=3$, then $SO(G_{2})-SO(G_{1})=\sqrt{5}+5-\sqrt{10}-\sqrt{20}<0$.

If $d_{G_{0}}(x)=2$ and $d_{G_{0}}(y)=2$, then $SO(G_{2})-SO(G_{1})=\sqrt{5}+5-\sqrt{17}-\sqrt{13}<0$.

If $d_{G_{0}}(x)=2$ and $d_{G_{0}}(y)=3$, then $SO(G_{2})-SO(G_{1})=\sqrt{5}+4\sqrt{2}-\sqrt{17}-\sqrt{20}<0$.

\noindent {\bf Case 2}. $l\geq 2$.
\begin{align*}
SO(G_{2})-SO(G_{1})
=&\sqrt{2^{2}+2^{2}}+\sqrt{(d_{G_{0}}(x)+2)^{2}+(d_{G_{0}}(y)+1)^{2}}\\
& -[\sqrt{(d_{G_{0}}(x)+2)^{2}+2^{2}}+\sqrt{(d_{G_{0}}(y)+1)^{2}+2^{2}}].
\end{align*}
Since $d_{G_{0}}(x)=1$ or $2$ and $d_{G_{0}}(y)=2$ or $3$.

If $d_{G_{0}}(x)=1$ and $d_{G_{0}}(y)=2$, then $SO(G_{2})-SO(G_{1})=2\sqrt{2}+3\sqrt{2}-\sqrt{13}-\sqrt{13}<0$.

If $d_{G_{0}}(x)=1$ and $d_{G_{0}}(y)=3$, then $SO(G_{2})-SO(G_{1})=2\sqrt{2}+5-\sqrt{13}-\sqrt{20}<0$.

If $d_{G_{0}}(x)=2$ and $d_{G_{0}}(y)=2$, then $SO(G_{2})-SO(G_{1})=2\sqrt{2}+5-\sqrt{20}-\sqrt{13}<0$.

If $d_{G_{0}}(x)=2$ and $d_{G_{0}}(y)=3$, then $SO(G_{2})-SO(G_{1})=2\sqrt{2}+4\sqrt{2}-\sqrt{20}-\sqrt{20}<0$.

Therefore $SO(G_{1})> SO(G_{2})$.
In a similar way, we also have $SO_{red}(G_{1})> SO_{red}(G_{2})$.
This completes the proof.
\end{proof}


\begin{lemma}\label{l-21}
Let $G_{0}$ be a connected graph with the vertex $x$ $($$d_{G_{0}}(x)=1$ or $2$$)$. Suppose that $P_{1}=u_{1}u_{2}\cdots u_{k}$ and $P_{2}=v_{1}v_{2}\cdots v_{l}$ are two paths. Denote by $G_{1}$ the graph gotten from $G_{0}$, $P_{1}$ and $P_{2}$ by attaching vertices $u_{1}x$ and $v_{1}x$.
Let $G_{2}=G_{1}-u_{1}x+u_{1}v_{l}$. Then $SO(G_{1})> SO(G_{2})$ and $SO_{red}(G_{1})> SO_{red}(G_{2})$.
\end{lemma}
\begin{proof}
Let $d_{G_{0}}(x)=t=1$ or $2$, $N_{G_{0}}(x)=\{z_{1},z_{2},\cdots,z_{t}\}$, $d_{G_{0}}(z_{i})=d_{i}$, $1\leq i\leq t$.

\noindent {\bf Case 1}. $k=l=1$.
\begin{align*}
SO(G_{2})-SO(G_{1})
=&\sqrt{1^{2}+2^{2}}+\sqrt{(t+1)^{2}+2^{2}}+\sum_{i=1}^{t}\sqrt{d_{i}^{2}+(t+1)^{2}}\\
 &-[\sqrt{(t+2)^{2}+1^{2}}+\sqrt{(t+2)^{2}+1^{2}}+\sum_{i=1}^{t}\sqrt{d_{i}^{2}+(t+2)^{2}}]\\
<&\sqrt{5}+\sqrt{(t+1)^{2}+2^{2}}-2\sqrt{(t+2)^{2}+1^{2}}.
\end{align*}

If $t=1$, then $SO(G_{2})-SO(G_{1})<\sqrt{5}+2\sqrt{2}-2\sqrt{10}<0$.

If $t=2$, then $SO(G_{2})-SO(G_{1})<\sqrt{5}+\sqrt{13}-2\sqrt{17}<0$.

\noindent {\bf Case 2}. $k=1, l\geq 2$.
\begin{align*}
SO(G_{2})-SO(G_{1})
<&2\sqrt{2}+\sum_{i=1}^{t}\sqrt{d_{i}^{2}+(t+1)^{2}}-[\sqrt{(t+2)^{2}+1^{2}}+\sum_{i=1}^{t}\sqrt{d_{i}^{2}+(t+2)^{2}}]\\
<&2\sqrt{2}-\sqrt{(t+2)^{2}+1^{2}}.
\end{align*}

If $t=1$, then $SO(G_{2})-SO(G_{1})<2\sqrt{2}-\sqrt{10}<0$.

If $t=2$, then $SO(G_{2})-SO(G_{1})<2\sqrt{2}-\sqrt{17}<0$.

\noindent {\bf Case 3}. $k\geq 2, l=1$.

The conclusion holds from Case 2 and symmetry.

\noindent {\bf Case 4}. $k\geq 2, l\geq 2$.
\begin{align*}
SO(G_{2})-SO(G_{1})
=&2\sqrt{2^{2}+2^{2}}+\sqrt{(t+1)^{2}+2^{2}}+\sum_{i=1}^{t}\sqrt{d_{i}^{2}+(t+1)^{2}}\\
 &-[\sqrt{1^{2}+2^{2}}+2\sqrt{(t+2)^{2}+2^{2}}+\sum_{i=1}^{t}\sqrt{d_{i}^{2}+(t+2)^{2}}]\\
<&4\sqrt{2}+\sqrt{(t+1)^{2}+2^{2}}-\sqrt{5}-2\sqrt{(t+2)^{2}+2^{2}}.
\end{align*}

If $t=1$, then $SO(G_{2})-SO(G_{1})<6\sqrt{2}-2\sqrt{13}-\sqrt{5}<0$.

If $t=2$, then $SO(G_{2})-SO(G_{1})<\sqrt{13}+4\sqrt{2}-5\sqrt{5}<0$.

Therefore $SO(G_{1})> SO(G_{2})$.
In a similar way, we also have $SO_{red}(G_{1})> SO_{red}(G_{2})$.
This completes the proof.
\end{proof}


\begin{lemma}\label{l-24}
Let $G_{0}$ be a connected graph with vertices $x$ and $y$ $($ $d_{G_{0}}(x)=2$ or $3$, $d_{G_{0}}(y)=2$ or $3$$)$. Suppose that $P_{1}=u_{1}u_{2}\cdots u_{k}$ and $P_{2}=v_{1}v_{2}\cdots v_{l}$ are two paths. Denote by $G_{1}$ the graph gotten from $G_{0}$, $P_{1}$ and $P_{2}$ by attaching vertices $u_{1}x$ and $v_{1}y$. Let $G_{2}=G_{1}-u_{1}x+u_{1}v_{l}$. Then $SO(G_{1})> SO(G_{2})$ and $SO_{red}(G_{1})> SO_{red}(G_{2})$.
\end{lemma}
\begin{proof}
Let $d_{G_{0}}(x)=t=2$ or $3$, $N_{G_{0}}(x)=\{z_{1},z_{2},\cdots,z_{t}\}$, $d_{G_{0}}(z_{i})=d_{i}$, $1\leq i\leq t$.

\noindent {\bf Case 1}. $k=l=1$.
\begin{align*}
SO(G_{2})-SO(G_{1})
=&\sqrt{1^{2}+2^{2}}+\sqrt{(d_{G_{0}}(y)+1)^{2}+2^{2}}+\sum_{i=1}^{t}\sqrt{d_{i}^{2}+t^{2}}\\
 &-[\sqrt{(d_{G_{0}}(y)+1)^{2}+1^{2}}+\sqrt{(t+1)^{2}+1^{2}}+\sum_{i=1}^{t}\sqrt{d_{i}^{2}+(t+1)^{2}}]\\
< &\sqrt{5}+\sqrt{(d_{G_{0}}(y)+1)^{2}+4}-[\sqrt{(d_{G_{0}}(y)+1)^{2}+1^{2}}+\sqrt{(t+1)^{2}+1^{2}}].
\end{align*}
Since $d_{G_{0}}(y)=2$ or $3$ and $t=2$ or $3$.

If $d_{G_{0}}(y)=2$ and $t=2$, then $SO(G_{2})-SO(G_{1})<\sqrt{5}+\sqrt{13}-\sqrt{10}-\sqrt{10}<0$.

If $d_{G_{0}}(y)=2$ and $t=3$, then $SO(G_{2})-SO(G_{1})<\sqrt{5}+\sqrt{13}-\sqrt{10}-\sqrt{17}<0$.

If $d_{G_{0}}(y)=3$ and $t=2$, then $SO(G_{2})-SO(G_{1})<\sqrt{5}+\sqrt{20}-\sqrt{10}-\sqrt{17}<0$.

If $d_{G_{0}}(y)=3$ and $t=3$, then $SO(G_{2})-SO(G_{1})<\sqrt{5}+\sqrt{20}-\sqrt{17}-\sqrt{17}<0$.

\noindent {\bf Case 2}. $k=1, l\geq 2$.
\begin{align*}
SO(G_{2})-SO(G_{1})
=&\sqrt{2^{2}+2^{2}}+\sum_{i=1}^{t}\sqrt{d_{i}^{2}+t^{2}}-[\sqrt{(t+1)^{2}+1^{2}}+\sum_{i=1}^{t}\sqrt{d_{i}^{2}+(t+1)^{2}}]\\
< &2\sqrt{2}-\sqrt{(t+1)^{2}+1^{2}}.
\end{align*}
Since $t=2$ or $3$.

If $t=2$, then $SO(G_{2})-SO(G_{1})<2\sqrt{2}-\sqrt{10}<0$.

If $t=3$, then $SO(G_{2})-SO(G_{1})<2\sqrt{2}-\sqrt{17}<0$.

\noindent {\bf Case 3}. $k\geq 2, l=1$.
\begin{align*}
SO(G_{2})-SO(G_{1})
=&\sqrt{2^{2}+2^{2}}+\sqrt{(d_{G_{0}}(y)+1)^{2}+2^{2}}+\sum_{i=1}^{t}\sqrt{d_{i}^{2}+t^{2}}\\
 &-[\sqrt{(d_{G_{0}}(y)+1)^{2}+1^{2}}+\sqrt{(t+1)^{2}+2^{2}}+\sum_{i=1}^{t}\sqrt{d_{i}^{2}+(t+1)^{2}}]\\
< &2\sqrt{2}+\sqrt{(d_{G_{0}}(y)+1)^{2}+4}-[\sqrt{(d_{G_{0}}(y)+1)^{2}+1^{2}}+\sqrt{(t+1)^{2}+4}].
\end{align*}
Since $d_{G_{0}}(y)=2$ or $3$ and $t=2$ or $3$.

If $d_{G_{0}}(y)=2$ and $t=2$, then $SO(G_{2})-SO(G_{1})<2\sqrt{2}+\sqrt{13}-\sqrt{10}-\sqrt{13}<0$.

If $d_{G_{0}}(y)=2$ and $t=3$, then $SO(G_{2})-SO(G_{1})<2\sqrt{2}+\sqrt{13}-\sqrt{10}-\sqrt{20}<0$.

If $d_{G_{0}}(y)=3$ and $t=2$, then $SO(G_{2})-SO(G_{1})<2\sqrt{2}+\sqrt{20}-\sqrt{17}-\sqrt{13}<0$.

If $d_{G_{0}}(y)=3$ and $t=3$, then $SO(G_{2})-SO(G_{1})<2\sqrt{2}+\sqrt{20}-\sqrt{17}-\sqrt{20}<0$.

\noindent {\bf Case 4}. $k\geq 2, l\geq 2$.
\begin{align*}
SO(G_{2})-SO(G_{1})
=&2\sqrt{2^{2}+2^{2}}+\sum_{i=1}^{t}\sqrt{d_{i}^{2}+t^{2}}\\
 &-[\sqrt{1^{2}+2^{2}}+\sqrt{(t+1)^{2}+2^{2}}+\sum_{i=1}^{t}\sqrt{d_{i}^{2}+(t+1)^{2}}]\\
<&4\sqrt{2}-\sqrt{5}-\sqrt{(t+1)^{2}+4}.
\end{align*}

Since $t=2$ or $3$.

If $t=2$, then $SO(G_{2})-SO(G_{1})< 4\sqrt{2}-\sqrt{5}-\sqrt{13}<0$.

If $t=3$, then $SO(G_{2})-SO(G_{1})< 4\sqrt{2}-\sqrt{5}-\sqrt{20}<0$.

Therefore $SO(G_{1})> SO(G_{2})$.
In a similar way, we also have $SO_{red}(G_{1})> SO_{red}(G_{2})$.
This completes the proof.
\end{proof}

\ \notag\

\section{Main results}

Denote by $CT_{n}$, $CU_{n}$, $CB_{n}$, $CTG_{n}$ the set of chemical trees, chemical unicyclic graphs, chemical bicyclic graphs, chemical tricyclic graphs with $n$ vertices, respectively. If $G$ has $t_{i}$ vertices of degree $d_{i}$ for $1\leq i\leq s$, then we denote the degree sequence of $G$ as $D(G)\triangleq (d_{1}^{t_{1}}, d_{2}^{t_{2}}, \cdots, d_{s}^{t_{s}})$, where $\sum\limits_{i=1}^{s}t_{i}=n$.

\subsection{Chemical trees}

Let $\Phi(n)=\{T\in (4^{1},2^{n-5},1^{4})|m_{1,2}(T)=m_{2,4}(T)=4, m_{1,4}(T)=0, m_{2,2}(T)=n-9\}$, $n\geq 9$, and $\Omega(n)=\{T\in (3^{3},2^{n-8},1^{5})|m_{1,2}(T)=m_{2,3}(T)=5, m_{1,3}(T)=0, m_{3,3}(T)=2, m_{2,2}(T)=n-13\}$, $n\geq 13$.

If $T\in \Phi(n)$, then
$$SO(T)=2\sqrt{2}n+12\sqrt{5}-18\sqrt{2}\approx 2\sqrt{2}n+1.376971607,\eqno{(1)}$$
$$SO_{red}(T)=\sqrt{2}n+4+4\sqrt{10}-9\sqrt{2}\approx \sqrt{2}n+3.921188579.\eqno{(2)}$$
\indent If $T\in \Omega(n)$, then
$$SO(T)=2\sqrt{2}n+5\sqrt{5}+5\sqrt{13}-20\sqrt{2} \approx 2\sqrt{2}n+0.923825017,\eqno{(3)}$$
$$SO_{red}(T)=\sqrt{2}n+5+5\sqrt{5}-9\sqrt{2} \approx \sqrt{2}n+3.452417826.\eqno{(4)}$$
Recall that $n_{i}$ is the numbers of vertices with degree $i$ in $G$.
If $G\in CT_{n}$, since $\sum_{i=1}^{4}n_{i}=n$ and $n_{1}+2n_{2}+3n_{3}+4n_{4}=2(n-1)$, then $n_{1}=n_{3}+2n_{4}+2$ and $n_{2}=n-2n_{3}-3n_{4}-2$.

\begin{theorem}\label{t-32}
Let $T^{*}\in CT_{n}$ $($$n\geq 9$$)$, $\Delta(T^{*})=4$. If $T^{*}\notin \Phi(n)$, then there exists $T\in \Phi(n)$, such that $SO(T)<SO(T^{*})$ and $SO_{red}(T)<SO_{red}(T^{*})$.
\end{theorem}
\begin{proof}
We consider the following two cases.

\noindent {\bf Case 1}. If $T^{*}\in (4^{1},2^{n-5},1^{4})$, then $T^{*}$ meets at least one of the following conditions $m_{1,2}(T^{*})\neq 4$, $m_{2,4}(T^{*})\neq 4$, $m_{1,4}(T^{*})\neq 0$, $m_{2,2}(T^{*})\neq n-9$, i.e., $m_{1,2}(T^{*})< 4$, $m_{2,4}(T^{*})< 4$, $m_{1,4}(T^{*})> 0$, $m_{2,2}(T^{*})> n-9$. By the transformation of Lemma \ref{l-22}, we can obtain a chemical tree $T\in \Phi(n)$, so we have $SO(T)<SO(T^{*})$.

\noindent {\bf Case 2}. If $T^{*}\notin (4^{1},2^{n-5},1^{4})$. By the transformation of Lemma \ref{l-21}, we can obtain a chemical tree $T\in (4^{1},2^{n-5},1^{4})$.
If $T\in \Phi(n)$, by Lemma \ref{l-21}, we have $SO(T)<SO(T^{*})$. If $T\notin \Phi(n)$, we are back to Case $1$ and the conclusion holds.

Therefore $SO(T)<SO(T^{*})$.
In a similar way, we also have $SO_{red}(T)<SO_{red}(T^{*})$.
This completes the proof.
\end{proof}

\begin{theorem}\label{t-33}
Let $T^{*}\in CT_{n}$ $($$n\geq 13$$)$, $\Delta(T^{*})=3$, $n_{3}(T^{*})\geq 3$. If $T^{*}\notin \Omega(n)$, then there exists $T\in \Omega(n)$, such that $SO(T)<SO(T^{*})$ and $SO_{red}(T)<SO_{red}(T^{*})$.
\end{theorem}
\begin{proof}
We consider the following two cases.

\noindent {\bf Case 1}. If $T^{*}\in (3^{3},2^{n-8},1^{5})$, then $T^{*}$ meets at least one of the following conditions $m_{1,2}(T^{*})\neq 5$, $m_{2,3}(T^{*})\neq 5$, $m_{1,3}(T^{*})\neq 0$, $m_{3,3}(T^{*})\neq 2$, $m_{2,2}(T^{*})\neq n-13$, i.e., $m_{1,2}(T^{*})< 5$, $m_{2,3}(T^{*})> 5$, $m_{1,3}(T^{*})> 0$, $m_{3,3}(T^{*})< 2$, $m_{2,2}(T^{*})> n-13$. By the transformation of Lemma \ref{l-22} and Lemma \ref{l-23}, we can obtain a chemical tree $T\in \Omega(n)$, so we have $SO(T)<SO(T^{*})$.

\noindent {\bf Case 2}. If $T^{*}\notin (3^{3},2^{n-8},1^{5})$, since $n_{1}(T^{*})=n_{3}(T^{*})+2$ and $n_{3}(T^{*})\geq 3$, then $n_{3}(T^{*})\geq 4$. By the transformation of Lemma \ref{l-21}, we can obtain a chemical tree $T\in (3^{3},2^{n-8},1^{5})$.
If $T\in \Omega(n)$, by Lemma \ref{l-21}, we have $SO(T)<SO(T^{*})$. If $T\notin \Omega(n)$, we are back to Case $1$ and the conclusion holds.

Therefore $SO(T)<SO(T^{*})$.
In a similar way, we also have $SO_{red}(T)<SO_{red}(T^{*})$.
This completes the proof.
\end{proof}

In what follows, we determine the extremal chemical trees with respect to (reduced) Sombor index.
It is worth noting that the relevant data of Table $1\sim 8$ except the values of (reduced) Sombor indices are from \cite{ghas2018,ggda2017}.

\begin{table}[h]
	\centering
    \caption{$CT_{n}$ with $\Delta \leq 3$, $n_{3}\leq 2$ and their (reduced)Sombor index.}
    \setlength{\tabcolsep}{1mm}{
	\begin{tabular}{c|ccccccc}\hline
	           &	 $m_{3,3}$   &	  $m_{2,3}$   &	  $m_{1,2}$   &	  $m_{1,3}$   &	  $m_{2,2}$   &	  $SO(G)$  &	$SO_{red}(G)$   \\ \hline
	$A_{1}$    &     $0$   &	  $0$   &	  $2$   &	  $0$   &	  $n-3$   &	  $(2\sqrt{2}n-4.013145419)$  &	  $(\sqrt{2}n-2.242640687)$  \\ \hline
    $A_{2}$    &     $0$   &	  $1$   &	  $1$   &	  $2$   &	  $n-5$   &	  $(2\sqrt{2}n-1.975961050)$  &	  $(\sqrt{2}n+0.165000165)$  \\ \hline
    $A_{3}$    &     $0$   &	  $2$   &	  $2$   &	  $1$   &	  $n-6$   &	  $(2\sqrt{2}n-2.125046582)$  &	  $(\sqrt{2}n-0.013145419)$  \\ \hline
    $A_{4}$    &     $0$   &	  $3$   &	  $3$   &	  $0$   &	  $n-7$   &	  $(2\sqrt{2}n-2.274132114)$  &	  $(\sqrt{2}n-0.191291004)$  \\ \hline
    $A_{5}$    &     $0$   &	  $2$   &	  $0$   &	  $4$   &	  $n-7$   &	  $(2\sqrt{2}n+0.061223318)$  &	  $(\sqrt{2}n+2.572641018)$  \\ \hline
    $A_{6}$    &     $0$   &	  $3$   &	  $1$   &	  $3$   &	  $n-8$   &	  $(2\sqrt{2}n-0.087862213)$  &	  $(\sqrt{2}n+2.394495433)$  \\ \hline
    $A_{7}$    &     $0$   &	  $4$   &	  $2$   &	  $2$   &	  $n-9$   &	  $(2\sqrt{2}n-0.236947745)$  &	  $(\sqrt{2}n+2.216349848)$  \\ \hline
    $A_{8}$    &     $1$   &	  $1$   &	  $1$   &	  $3$   &	  $n-7$   &	  $(2\sqrt{2}n-0.227896952)$  &	  $(\sqrt{2}n+2.165000165)$  \\ \hline
    $A_{9}$    &     $0$   &	  $5$   &	  $3$   &	  $1$   &	  $n-10$  &	  $(2\sqrt{2}n-0.386033277)$  &	  $(\sqrt{2}n+2.038204263)$  \\ \hline
    $A_{10}$   &     $1$   &	  $2$   &	  $2$   &	  $2$   &	  $n-8$   &	  $(2\sqrt{2}n-0.376982484)$  &	  $(\sqrt{2}n+1.986854580)$  \\ \hline
    $A_{11}$   &     $0$   &	  $6$   &	  $4$   &	  $0$   &	  $n-11$  &	  $(2\sqrt{2}n-0.535118809)$  &	  $(\sqrt{2}n+1.860058678)$  \\ \hline
    $A_{12}$   &     $1$   &	  $3$   &	  $3$   &	  $1$   &	  $n-9$   &	  $(2\sqrt{2}n-0.526068016)$  &	  $(\sqrt{2}n+1.808708995)$  \\ \hline
    $A_{13}$   &     $1$   &	  $4$   &	  $4$   &	  $0$   &	  $n-10$  &	  $(2\sqrt{2}n-0.675153548)$  &	  $(\sqrt{2}n+1.630563411)$  \\ \hline
	\end{tabular}}
	
	\label{table1}
\end{table}

\begin{theorem}\label{t-34}
If $n\geq 13$, $T_{1}\in A_{1}$, $T_{2}\in A_{4}$, $T_{3}\in A_{3}$, $T_{4}\in A_{2}$, $T_{5}\in A_{13}$, $T_{6}\in A_{11}$, $T_{7}\in A_{12}$, $T_{8}\in A_{9}$, $T_{9}\in A_{10}$, $T_{10}\in A_{7}$, $T_{11}\in A_{8}$, $T_{12}\in A_{6}$, $T_{13}\in A_{5}$, $T_{14}\in \Omega_{n}$, and $T\in CT_{n}\setminus \{T_{1},T_{2},\cdots, T_{14}\}$, then
$SO(T_{1})<SO(T_{2})<SO(T_{3})<SO(T_{4})<SO(T_{5})<SO(T_{6})<SO(T_{7})<SO(T_{8})<SO(T_{9})<SO(T_{10})<SO(T_{11})<SO(T_{12})<SO(T_{13})<SO(T_{14})<SO(T)$.
\end{theorem}
\begin{proof}
By Table \ref{table1} and the Sombor index of chemical trees among $\Omega(n)$, we have $SO(T_{1})<SO(T_{2})<SO(T_{3})<SO(T_{4})<SO(T_{5})<SO(T_{6})<SO(T_{7})<SO(T_{8})<SO(T_{9})<SO(T_{10})<SO(T_{11})<SO(T_{12})<SO(T_{13})<SO(T_{14})$.

If $\Delta(T)\leq 3$ and $n_{3}(T)\leq 2$, the conclusion holds. If $\Delta(T)=3$ and $n_{3}(T)\geq 3$, then by Theorem \ref{t-33}, the conclusion holds.
If $\Delta(T)= 4$, then by Equation (1),(3) and Theorem \ref{t-32}, the conclusion holds.
\end{proof}

Similar to the proof of Theorem \ref{t-34}, we have
\begin{theorem}\label{t-341}
If $n\geq 13$, $T_{1}\in A_{1}$, $T_{2}\in A_{4}$, $T_{3}\in A_{3}$, $T_{4}\in A_{2}$, $T_{5}\in A_{13}$, $T_{6}\in A_{12}$, $T_{7}\in A_{11}$, $T_{8}\in A_{10}$, $T_{9}\in A_{9}$, $T_{10}\in A_{8}$, $T_{11}\in A_{7}$, $T_{12}\in A_{6}$, $T_{13}\in A_{5}$, $T_{14}\in \Omega_{n}$, and $T\in CT_{n}\setminus \{T_{1},T_{2},\cdots, T_{14}\}$, then
$SO_{red}(T_{1})<SO_{red}(T_{2})<SO_{red}(T_{3})<SO_{red}(T_{4})<SO_{red}(T_{5})<SO_{red}(T_{6})<SO_{red}(T_{7})<SO_{red}(T_{8})<SO_{red}(T_{9})<
SO_{red}(T_{10})<SO_{red}(T_{11})<SO_{red}(T_{12})$ $<SO_{red}(T_{13})<SO_{red}(T_{14})<SO_{red}(T)$.
\end{theorem}

\subsection{Chemical unicyclic graphs}

In this subsection, we consider the extremal chemical unicyclic graphs with respect to (reduced) Sombor index.

\begin{table}[h]
	\centering
    \caption{Degree distributions $(DD)$ of $CU_{n}$ with $n_{1}\leq 2$.}
	\begin{tabular}{c|cccc}\hline
	           &	 $n_{4}$   &  $n_{3}$   & $n_{2}$   & $n_{1}$   \\ \hline
	$H_{1}$    &     $0$   &	  $0$   &	  $n$   &	  $0$        \\ \hline 
    $H_{2}$    &     $0$   &	  $1$   &	  $n-2$   &	  $1$        \\ \hline
    $H_{3}$    &     $1$   &	  $0$   &	  $n-3$   &	  $2$        \\ \hline
    $H_{4}$    &     $0$   &	  $2$   &	  $n-4$   &	  $2$        \\ \hline
	\end{tabular}
	
	\label{table2}
\end{table}

\begin{table}[h]
	\centering
    \caption{$CU_{n}$ with $n_{1}\leq 2$ and their (reduced)Sombor index.}
    \setlength{\tabcolsep}{1mm}{
	\begin{tabular}{c|cccccccccc}\hline
	            &  $DD$   &	 $m_{1,2}$   &  $m_{1,3}$   & $m_{1,4}$   &	 $m_{2,3}$  &  $m_{2,4}$  &  $m_{3,3}$  & $m_{2,2}$   &	 $SO(G)$ &	 $SO_{red}(G)$   \\ \hline
	$\alpha_{1}$    &  $H_{1}$   &    $0$   &	  $0$  &      $0$   &     $0$   &	  $0$  &    $0$   &     $n$   &   $(2\sqrt{2}n)$ &   $(\sqrt{2}n)$ \\
    \hline
    $\alpha_{2}$    &  $H_{2}$   &	  $0$   &	  $1$  &	  $0$   &	  $2$   &	  $0$  &	  $0$   &	  $n-3$   &	  $(2\sqrt{2}n+1.888)$ &	  $(\sqrt{2}n+2.229)$ \\ \hline

    $\alpha_{3}$    &  $H_{2}$   &	  $1$   &	  $0$  &	  $0$   &	  $3$   &	  $0$  &	  $0$   &	  $n-4$   &	  $(2\sqrt{2}n+1.739)$  &	  $(\sqrt{2}n+2.051)$ \\ \hline

    $\alpha_{4}$    &  $H_{3}$   &	  $0$   &	  $0$  &	  $2$   &	  $0$   &	  $2$  &	  $0$   &	  $n-4$   &	  $(2\sqrt{2}n+5.876)$  &	  $(\sqrt{2}n+6.667)$ \\ \hline

    $\alpha_{5}$    &  $H_{3}$   &	  $1$   &	  $0$  &	  $1$   &	  $0$   &	  $3$  &	  $0$   &	  $n-5$   &	  $(2\sqrt{2}n+5.633)$  &	  $(\sqrt{2}n+6.415)$ \\ \hline

    $\alpha_{6}$    &  $H_{3}$   &	  $2$   &	  $0$  &	  $0$   &	  $0$   &	  $4$  &	  $0$   &	  $n-6$   &	  $(2\sqrt{2}n+5.390)$  &	  $(\sqrt{2}n+6.163)$ \\ \hline

    $\alpha_{7}$    &  $H_{4}$   &	  $0$   &	  $2$  &	  $0$   &	  $2$   &	  $0$  &	  $1$   &	  $n-5$   &	  $(2\sqrt{2}n+3.636)$  &	  $(\sqrt{2}n+4.229)$  \\ \hline

    $\alpha_{8}$    &  $H_{4}$   &	  $1$   &	  $1$  &	  $0$   &	  $3$   &	  $0$  &	  $1$   &	  $n-6$   &	  $(2\sqrt{2}n+3.487)$  &	  $(\sqrt{2}n+4.051)$ \\ \hline

    $\alpha_{9}$    &  $H_{4}$   &	  $2$   &	  $0$  &	  $0$   &	  $4$   &	  $0$  &	  $1$   &	  $n-7$  &	  $(2\sqrt{2}n+3.337)$  &	  $(\sqrt{2}n+3.873)$ \\ \hline

    $\alpha_{10}$   &  $H_{4}$   &	  $0$   &	  $2$  &	  $0$   &	  $4$   &	  $0$  &	  $0$   &	  $n-6$   &	  $(2\sqrt{2}n+3.776)$  &	  $(\sqrt{2}n+4.458)$  \\ \hline

    $\alpha_{11}$   &  $H_{4}$   &	  $1$   &	  $1$  &	  $0$   &	  $5$   &	  $0$  &	  $0$   &	  $n-7$  &	  $(2\sqrt{2}n+3.627)$  &	  $(\sqrt{2}n+4.280)$ \\ \hline

    $\alpha_{12}$   &  $H_{4}$   &	  $2$   &	  $0$  &	  $0$   &	  $6$   &	  $0$  &	  $0$   &	  $n-8$   &	  $(2\sqrt{2}n+3.478)$  &	  $(\sqrt{2}n+4.102)$
    \\ \hline
	\end{tabular}}
	
	\label{table3}
\end{table}

\begin{lemma}\label{l-35}\cite{ghas2018,ggda2017}
$G\in CU_{n}$ and $n_{1}(G)\leq 2$ if and only if $G$ belongs to one of equivalence classes given in Table \ref{table2}.
\end{lemma}

\begin{theorem}\label{t-36}
If $n\geq 7$, $G_{1}\in \alpha_{1}$, $G_{2}\in \alpha_{3}$, $G_{3}\in \alpha_{2}$, $G_{4}\in \alpha_{9}$ in Table \ref{table3}. $G\in CU_{n}\setminus \{G_{1},G_{2},G_{3},G_{4}\}$, then
$SO(G_{1})<SO(G_{2})<SO(G_{3})<SO(G_{4})<SO(G)$.
\end{theorem}
\begin{proof}
By Table \ref{table3}, we have $SO(G_{1})<SO(G_{2})<SO(G_{3})<SO(G_{4})$.

If $n_{1}(G)\leq 2$, by Table \ref{table3}, the conclusion holds. If $n_{1}(G)\geq 3$, by the transformations of Lemma \ref{l-21} and Lemma \ref{l-24}, we can obtain a chemical unicyclic graphs $G^{*}$ with $n_{1}(G^{*})= 2$, so we have $SO(G)>SO(G^{*})$. By Table \ref{table3}, $SO(G_{4})\leq SO(G^{*})$. Thus, the conclusion holds.
\end{proof}

Similar to the proof of Theorem \ref{t-36}, we have
\begin{theorem}\label{t-361}
If $n\geq 7$, $G_{1}\in \alpha_{1}$, $G_{2}\in \alpha_{3}$, $G_{3}\in \alpha_{2}$, $G_{4}\in \alpha_{9}$ in Table \ref{table3}. $G\in CU_{n}\setminus \{G_{1},G_{2},G_{3},G_{4}\}$, then
$SO_{red}(G_{1})<SO_{red}(G_{2})<SO_{red}(G_{3})<SO_{red}(G_{4})<SO_{red}(G)$.
\end{theorem}

\subsection{Chemical bicyclic graphs}

In this subsection, we consider the extremal chemical bicyclic graphs with respect to (reduced) Sombor index.

\begin{table}[h]
	\centering
    \caption{Degree distributions $(DD)$ of $CB_{n}$ with $n_{1}\leq 1$.}
	\begin{tabular}{c|cccc}\hline
	           &	 $n_{4}$   &  $n_{3}$   & $n_{2}$   & $n_{1}$   \\ \hline
	$B_{1}$    &     $1$   &	  $0$   &	  $n-1$   &	  $0$        \\ \hline 
    $B_{2}$    &     $0$   &	  $2$   &	  $n-2$   &	  $0$        \\ \hline
    $B_{3}$    &     $1$   &	  $1$   &	  $n-3$   &	  $1$        \\ \hline
    $B_{4}$    &     $0$   &	  $3$   &	  $n-4$   &	  $1$        \\ \hline
	\end{tabular}
	
	\label{table4}
\end{table}

\begin{table}[h]
	\centering
    \caption{$CB_{n}$ with $n_{1}\leq 1$ and their (reduced)Sombor index.}
    \setlength{\tabcolsep}{1mm}{
	\begin{tabular}{c|ccccccccccc}\hline
	       &  $DD$  & $m_{1,2}$  &	$m_{1,3}$   &  $m_{1,4}$   & $m_{2,3}$   &	 $m_{2,4}$  &  $m_{3,3}$  &  $m_{3,4}$  & $m_{2,2}$   &	 $SO(G)$ &	 $SO_{red}(G)$ \\
    \hline
	$\beta_{1}$    &  $B_{1}$  &    $0$  &    $0$   &	  $0$  &      $0$   &     $4$   &	  $0$  &    $0$   &     $n-3$   &     $(2\sqrt{2}n+9.403)$ &     $(\sqrt{2}n+8.406)$ \\

    \hline
    $\beta_{2}$    &  $B_{2}$  &    $0$  &	  $0$   &	  $0$  &	  $4$   &	  $0$   &	  $1$  &	  $0$   &	$n-4$   &	  $(2\sqrt{2}n+7.351)$ &     $(\sqrt{2}n+6.115)$ \\ \hline

    $\beta_{3}$    &  $B_{2}$  &    $0$  &	  $0$   &	  $0$  &	  $6$   &	  $0$   &	  $0$  &	  $0$   &	  $n-5$   &	  $(2\sqrt{2}n+7.491)$ &     $(\sqrt{2}n+6.345)$ \\ \hline

    $\beta_{4}$    &  $B_{3}$  &    $0$  &	  $0$   &	  $1$  &	  $2$   &	  $2$   &	  $0$  &	  $1$   &	  $n-5$   &	  $(2\sqrt{2}n+11.136)$ &     $(\sqrt{2}n+10.331)$ \\ \hline

    $\beta_{5}$    &  $B_{3}$  &    $1$  &	  $0$   &	  $0$  &	  $2$   &	  $3$   &	  $0$  &	  $1$   &	  $n-6$   &	  $(2\sqrt{2}n+10.893)$ &     $(\sqrt{2}n+10.079)$ \\ \hline

    $\beta_{6}$    &  $B_{3}$  &    $0$  &	  $0$   &	  $1$  &	  $3$   &	  $3$   &	  $0$  &	  $0$   &	  $n-6$   &	  $(2\sqrt{2}n+11.385)$ &     $(\sqrt{2}n+10.709)$ \\ \hline

    $\beta_{7}$    &  $B_{3}$  &    $1$  &	  $0$   &	  $0$  &	  $3$   &	  $4$   &	  $0$  &	  $0$   &	  $n-7$   &	  $(2\sqrt{2}n+11.142)$ &     $(\sqrt{2}n+10.457)$ \\ \hline

    $\beta_{8}$    &  $B_{4}$  &    $0$  &	  $1$   &	  $0$  &	  $2$   &	  $0$   &	  $3$  &	  $0$   &	  $n-5$   &	  $(2\sqrt{2}n+8.959)$ &     $(\sqrt{2}n+7.886)$ \\ \hline

    $\beta_{9}$    &  $B_{4}$  &    $1$  &	  $0$   &	  $0$  &	  $3$   &	  $0$   &	  $3$  &	  $0$   &	  $n-6$  &	  $(2\sqrt{2}n+8.810)$ &     $(\sqrt{2}n+7.708)$ \\ \hline

    $\beta_{10}$   &  $B_{4}$  &    $0$  &	  $1$   &	  $0$  &	  $4$   &	  $0$   &	  $2$  &	  $0$   &	  $n-6$   &	  $(2\sqrt{2}n+9.099)$ &     $(\sqrt{2}n+8.115)$ \\ \hline

    $\beta_{11}$   &  $B_{4}$  &    $1$  &	  $0$   &	  $0$  &	  $5$   &	  $0$   &	  $2$  &	  $0$   &	  $n-7$  &	  $(2\sqrt{2}n+8.950)$ &     $(\sqrt{2}n+7.937)$ \\ \hline

    $\beta_{12}$   &  $B_{4}$  &    $0$  &	  $1$   &	  $0$  &	  $6$   &	  $0$   &	  $1$  &	  $0$   &	  $n-7$   &	  $(2\sqrt{2}n+9.239)$ &     $(\sqrt{2}n+8.345)$ \\ \hline

    $\beta_{13}$   &  $B_{4}$  &    $1$  &	  $0$   &	  $0$  &	  $7$   &	  $0$   &	  $1$  &	  $0$   &	  $n-8$   &	  $(2\sqrt{2}n+9.090)$ &     $(\sqrt{2}n+8.167)$ \\ \hline

    $\beta_{14}$   &  $B_{4}$  &    $0$  &	  $1$   &	  $0$  &	  $8$   &	  $0$   &	  $0$  &	  $0$   &	  $n-8$  &	  $(2\sqrt{2}n+9.379)$ &     $(\sqrt{2}n+8.574)$ \\ \hline

    $\beta_{15}$   &  $B_{4}$  &    $1$  &	  $0$   &	  $0$  &	  $9$   &	  $0$   &	  $0$  &	  $0$   &	  $n-9$   &	  $(2\sqrt{2}n+9.230)$ &     $(\sqrt{2}n+8.396)$ \\ \hline
	\end{tabular}}
	
	\label{table5}
\end{table}

\begin{lemma}\label{l-37}\cite{ghas2018,ggda2017}
$G\in CB_{n}$ and $n_{1}(G)\leq 1$ if and only if $G$ belongs to one of equivalence classes given in Table \ref{table4}.
\end{lemma}

\begin{theorem}\label{t-38}
If $n\geq 6$, $G_{1}\in \beta_{2}$, $G_{2}\in \beta_{3}$, $G_{3}\in \beta_{9}$ in Table \ref{table5}. $G\in CB_{n}\setminus \{G_{1},G_{2},G_{3}\}$, then
$SO(G_{1})<SO(G_{2})<SO(G_{3})<SO(G)$.
\end{theorem}
\begin{proof}
By Table \ref{table5}, we have $SO(G_{1})<SO(G_{2})<SO(G_{3})$.

If $n_{1}(G)\leq 1$, by Table \ref{table5}, the conclusion holds. If $n_{1}(G)\geq 2$, by the transformations of Lemma \ref{l-21} and Lemma \ref{l-24}, we can obtain a chemical bicyclic graphs $G^{*}$ with $n_{1}(G^{*})= 1$, so we have $SO(G)>SO(G^{*})$. By Table \ref{table5}, $SO(G_{3})\leq SO(G^{*})$. Thus, the conclusion holds.
\end{proof}

Similar to the proof of Theorem \ref{t-38}, we have
\begin{theorem}\label{t-381}
If $n\geq 6$, $G_{1}\in \beta_{2}$, $G_{2}\in \beta_{3}$, $G_{3}\in \beta_{9}$ in Table \ref{table5}. $G\in CB_{n}\setminus \{G_{1},G_{2},G_{3}\}$, then
$SO_{red}(G_{1})<SO_{red}(G_{2})<SO_{red}(G_{3})<SO_{red}(G)$.
\end{theorem}

\subsection{Chemical tricyclic graphs}

In this subsection, we consider the extremal chemical tricyclic graphs with respect to (reduced) Sombor index.

\begin{table}[h]
	\centering
    \caption{Degree distributions $(DD)$ of $CTG_{n}$ with $n_{1}\leq 1$.}
	\begin{tabular}{c|cccc}\hline
	           &	 $n_{4}$   &  $n_{3}$   & $n_{2}$   & $n_{1}$   \\ \hline
	$E_{1}$    &     $2$   &	  $0$   &	  $n-2$   &	  $0$        \\ \hline 
    $E_{2}$    &     $1$   &	  $2$   &	  $n-3$   &	  $0$        \\ \hline
    $E_{3}$    &     $0$   &	  $4$   &	  $n-4$   &	  $0$        \\ \hline
    $E_{4}$    &     $2$   &	  $1$   &	  $n-4$   &	  $1$        \\ \hline
    $E_{5}$    &     $1$   &	  $3$   &	  $n-5$   &	  $1$        \\ \hline
    $E_{6}$    &     $0$   &	  $5$   &	  $n-6$   &	  $1$        \\ \hline
	\end{tabular}
	
	\label{table6}
\end{table}

\begin{table}[!htb]
	\centering
    \caption{$CTG_{n}$ with $n_{1}\leq 1$ and their (reduced)Sombor index.}
    \setlength{\tabcolsep}{0.5mm}{
	\begin{tabular}{c|cccccccccccc}\hline
	       & $DD$  & $m_{1,2}$ & $m_{1,3}$ & $m_{1,4}$ &  $m_{2,3}$   & $m_{2,4}$   &  $m_{3,3}$  &  $m_{3,4}$  &  $m_{4,4}$  & $m_{2,2}$   & $SO(G)$ & $SO_{red}(G)$\\
    \hline

	$\gamma_{1}$    &  $E_{1}$  &   $0$  &   $0$  &    $0$   &	  $0$  &      $8$   &     $0$   &	  $0$  &    $0$   &     $n-6$   & $(2\sqrt{2}n+18.806)$ & $(\sqrt{2}n+16.812)$ \\ \hline

    $\gamma_{2}$    &  $E_{1}$  &   $0$ &    $0$  &	  $0$   &	  $0$  &	  $6$   &	  $0$   &	  $0$  &	 $1$   &	$n-5$   &	 $(2\sqrt{2}n+18.347)$ & $(\sqrt{2}n+16.145)$  \\ \hline

    $\gamma_{3}$    &  $E_{2}$  &   $0$ &    $0$  &	  $0$   &	  $2$  &	  $2$   &	  $1$   &	  $2$  &	  $0$   &	  $n-5$   &	  $(2\sqrt{2}n+16.255)$ & $(\sqrt{2}n+13.765)$   \\ \hline

    $\gamma_{4}$    &  $E_{2}$  &   $0$ &    $0$  &	  $0$   &	  $3$  &	  $3$   &	  $1$   &	  $1$  &	  $0$   &	  $n-6$   &	  $(2\sqrt{2}n+16.505)$ & $(\sqrt{2}n+14.143)$  \\ \hline

    $\gamma_{5}$    &  $E_{2}$  &   $0$ &    $0$  &	  $0$   &	  $4$  &	  $4$   &	  $1$   &	  $0$  &	  $0$   &	  $n-7$   &	  $(2\sqrt{2}n+16.754)$ & $(\sqrt{2}n+14.522)$  \\ \hline

    $\gamma_{6}$    &  $E_{2}$  &   $0$ &    $0$  &	  $0$   &	  $4$  &	  $2$   &	  $0$   &	  $2$  &	  $0$   &	  $n-6$   &	  $(2\sqrt{2}n+16.395)$ & $(\sqrt{2}n+13.994)$  \\ \hline

    $\gamma_{7}$    &  $E_{2}$  &   $0$ &    $0$  &	  $0$   &	  $5$  &	  $3$   &	  $0$   &	  $1$  &	  $0$   &	  $n-7$   &	  $(2\sqrt{2}n+16.645)$ & $(\sqrt{2}n+14.373)$  \\ \hline

    $\gamma_{8}$    &  $E_{2}$  &   $0$ &    $0$  &	  $0$   &	  $6$  &	  $4$   &	  $0$   &	  $0$  &	  $0$   &	  $n-8$   &	  $(2\sqrt{2}n+16.894)$ & $(\sqrt{2}n+14.751)$  \\ \hline

    $\gamma_{9}$    &  $E_{3}$  &   $0$ &    $0$  &	  $0$   &	  $2$  &	  $0$   &	  $5$   &	  $0$  &	  $0$   &	  $n-5$  &	  $(2\sqrt{2}n+14.282)$ & $(\sqrt{2}n+11.543)$  \\ \hline

    $\gamma_{10}$   &  $E_{3}$  &   $0$ &    $0$  &	  $0$   &	  $4$  &	  $0$   &	  $4$   &	  $0$  &	  $0$   &	  $n-6$   &	  $(2\sqrt{2}n+14.422)$ & $(\sqrt{2}n+11.772)$  \\ \hline

    $\gamma_{11}$   &  $E_{3}$  &   $0$ &    $0$  &	  $0$   &	  $6$  &	  $0$   &	  $3$   &	  $0$  &	  $0$   &	  $n-7$  &	  $(2\sqrt{2}n+14.562)$ & $(\sqrt{2}n+12.002)$  \\ \hline

    $\gamma_{12}$   &  $E_{3}$  &   $0$ &    $0$  &	  $0$   &	  $8$  &	  $0$   &	  $2$   &	  $0$  &	  $0$   &	  $n-8$   &	  $(2\sqrt{2}n+14.702)$ & $(\sqrt{2}n+12.231)$  \\ \hline

    $\gamma_{13}$   &  $E_{3}$  &   $0$ &    $0$  &	  $0$   &	  $10$  &	  $0$   &	  $1$   &	  $0$  &	  $0$   &	  $n-9$   &	  $(2\sqrt{2}n+14.842)$ & $(\sqrt{2}n+12.461)$   \\ \hline

    $\gamma_{14}$   &  $E_{3}$  &   $0$ &    $0$  &	  $0$   &	  $12$  &	  $0$   &	  $0$   &	  $0$  &	  $0$   &	  $n-10$  &	  $(2\sqrt{2}n+14.982)$ & $(\sqrt{2}n+12.690)$  \\ \hline

    $\gamma_{15}$   &  $E_{4}$  &   $0$ &    $0$  &	  $1$   &	  $1$  &	  $3$   &	  $0$   &	  $2$  &	  $1$   &	  $n-6$   &	  $(2\sqrt{2}n+19.831)$ & $(\sqrt{2}n+17.691)$  \\ \hline

    $\gamma_{16}$    &  $E_{4}$  &   $1$ &    $0$  &    $0$   &	  $1$  &      $4$   &     $0$   &	  $2$  &    $1$   &     $n-7$   & $(2\sqrt{2}n+19.588)$ & $(\sqrt{2}n+17.439)$  \\ \hline

    $\gamma_{17}$    &  $E_{4}$  &   $0$ &    $0$  &	  $1$   &	  $2$  &	  $4$   &	  $0$   &	  $1$  &	  $1$   &	$n-7$   &	$(2\sqrt{2}n+20.080)$ & $(\sqrt{2}n+18.069)$  \\ \hline

    $\gamma_{18}$    &  $E_{4}$  &   $1$ &    $0$  &	  $0$   &	  $2$  &	  $5$   &	  $0$   &	  $1$  &	  $1$   &	  $n-8$   &	  $(2\sqrt{2}n+19.837)$  & $(\sqrt{2}n+17.818)$ \\ \hline

    $\gamma_{19}$    &  $E_{4}$  &   $0$ &    $0$  &	  $1$   &	  $3$  &	  $5$   &	  $0$   &	  $0$  &	  $1$   &	  $n-8$   &	  $(2\sqrt{2}n+20.329)$ & $(\sqrt{2}n+18.448)$  \\ \hline

    $\gamma_{20}$    &  $E_{4}$  &   $1$ &    $0$  &	  $0$   &	  $3$  &	  $6$   &	  $0$   &	  $0$  &	  $1$   &	  $n-9$   &	  $(2\sqrt{2}n+20.086)$  & $(\sqrt{2}n+18.196)$ \\ \hline

    $\gamma_{21}$    &  $E_{4}$  &   $0$ &    $0$  &	  $1$   &	  $1$  &	  $5$   &	  $0$   &	  $2$  &	  $0$   &	  $n-7$   &	  $(2\sqrt{2}n+20.290)$  & $(\sqrt{2}n+18.359)$ \\ \hline

    $\gamma_{22}$    &  $E_{4}$  &   $1$ &    $0$  &	  $0$   &	  $1$  &	  $6$   &	  $0$   &	  $2$  &	  $0$   &	  $n-8$   &	  $(2\sqrt{2}n+20.047)$  & $(\sqrt{2}n+18.107)$ \\ \hline

    $\gamma_{23}$    &  $E_{4}$  &   $0$ &    $0$  &	  $1$   &	  $2$  &	  $6$   &	  $0$   &	  $1$  &	  $0$   &	  $n-8$   &	  $(2\sqrt{2}n+20.539)$ & $(\sqrt{2}n+18.737)$  \\ \hline

    $\gamma_{24}$    &  $E_{4}$  &   $1$ &    $0$  &	  $0$   &	  $2$  &	  $7$   &	  $0$   &	  $1$  &	  $0$   &	  $n-9$  &	  $(2\sqrt{2}n+20.296)$  & $(\sqrt{2}n+18.485)$ \\ \hline

    $\gamma_{25}$   &  $E_{4}$  &   $0$ &    $0$  &	  $1$   &	  $3$  &	  $7$   &	  $0$   &	  $0$  &	  $0$   &	  $n-9$   &	  $(2\sqrt{2}n+20.788)$ & $(\sqrt{2}n+16.116)$  \\ \hline

    $\gamma_{26}$   &  $E_{4}$  &   $1$ &    $0$  &	  $0$   &	  $3$  &	  $8$   &	  $0$   &	  $0$  &	  $0$   &	  $n-10$  &	  $(2\sqrt{2}n+20.545)$ & $(\sqrt{2}n+18.864)$  \\ \hline

    $\gamma_{27}$   &  $E_{5}$  &   $0$ &    $0$  &	  $1$   &	  $2$  &	  $0$   &	  $2$   &	  $3$  &	  $0$   &	  $n-6$   &	  $(2\sqrt{2}n+17.848)$ & $(\sqrt{2}n+15.460)$  \\ \hline

    $\gamma_{28}$   &  $E_{5}$  &   $0$ &    $0$  &	  $1$   &	  $4$  &	  $0$   &	  $1$   &	  $3$  &	  $0$   &	  $n-7$   &	  $(2\sqrt{2}n+17.988)$ & $(\sqrt{2}n+15.689)$  \\ \hline

    $\gamma_{29}$   &  $E_{5}$  &   $0$ &    $0$  &	  $1$   &	  $6$  &	  $0$   &	  $0$   &	  $3$  &	  $0$   &	  $n-8$  &	  $(2\sqrt{2}n+18.128)$ & $(\sqrt{2}n+15.919)$  \\ \hline

    $\gamma_{30}$   &  $E_{5}$  &   $0$ &    $0$  &	  $1$   &	  $1$  &	  $1$   &	  $3$   &	  $2$  &	  $0$   &	  $n-6$   &	  $(2\sqrt{2}n+17.958)$ & $(\sqrt{2}n+15.609)$  \\ \hline

    $\gamma_{31}$    &  $E_{5}$  &   $0$ &    $0$  &    $1$   &	  $3$  &      $1$   &     $2$   &	  $2$  &    $0$   &     $n-7$   &
     $(2\sqrt{2}n+18.098)$ & $(\sqrt{2}n+15.838)$  \\ \hline

    $\gamma_{32}$    &  $E_{5}$  &   $0$ &    $0$  &	  $1$   &	  $5$  &	  $1$   &	  $1$   &	  $2$  &	  $0$   &	$n-8$   &	  $(2\sqrt{2}n+18.238)$  & $(\sqrt{2}n+16.068)$ \\ \hline

    $\gamma_{33}$    &  $E_{5}$  &   $0$  &    $0$  &	  $1$   &	  $7$  &	  $1$   &	  $0$   &	  $2$  &	  $0$   &	  $n-9$   &	  $(2\sqrt{2}n+18.378)$  & $(\sqrt{2}n+16.297)$ \\ \hline

    $\gamma_{34}$    &  $E_{5}$  &   $0$ &    $0$  &	  $1$   &	  $2$  &	  $2$   &	  $3$   &	  $1$  &	  $0$   &	  $n-7$   &	  $(2\sqrt{2}n+18.207)$  & $(\sqrt{2}n+15.988)$ \\ \hline

    $\gamma_{35}$    &  $E_{5}$  &   $0$ &    $0$  &	  $1$   &	  $4$  &	  $2$   &	  $2$   &	  $1$  &	  $0$   &	  $n-8$   &	  $(2\sqrt{2}n+18.347)$ & $(\sqrt{2}n+16.217)$  \\ \hline

    $\gamma_{36}$    &  $E_{5}$  &   $0$  &    $0$  &	  $1$   &	  $6$  &	  $2$   &	  $1$   &	  $1$  &	  $0$   &	  $n-9$   &	  $(2\sqrt{2}n+18.487)$  & $(\sqrt{2}n+16.447)$ \\ \hline

    $\gamma_{37}$    &  $E_{5}$  &   $0$  &   $0$  &    $1$   &	  $8$  &      $2$   &     $0$   &	  $1$  &    $0$   &     $n-10$   &
     $(2\sqrt{2}n+18.627)$ & $(\sqrt{2}n+16.676)$ \\ \hline

    $\gamma_{38}$    &  $E_{5}$  &   $0$ &    $0$  &	  $1$   &	  $3$  &	  $3$   &	  $3$   &	  $0$  &	 $0$   &	$n-8$   &	  $(2\sqrt{2}n+18.456)$ & $(\sqrt{2}n+16.366)$ \\ \hline

    $\gamma_{39}$    &  $E_{5}$  &   $0$ &    $0$  &	  $1$   &	  $5$  &	  $3$   &	  $2$   &	  $0$  &	  $0$   &	  $n-9$   &	  $(2\sqrt{2}n+18.596)$ & $(\sqrt{2}n+16.596)$ \\ \hline

    $\gamma_{40}$    &  $E_{5}$  &   $0$ &    $0$  &	  $1$   &	  $7$  &	  $3$   &	  $1$   &	  $0$  &	  $0$   &	  $n-10$   &	  $(2\sqrt{2}n+18.736)$ & $(\sqrt{2}n+16.825)$ \\ \hline

	\end{tabular}}
	
	\label{table7}
\end{table}

\begin{table}[!htb]
	\centering
    \caption{$CTG_{n}$ with $n_{1}\leq 1$ and their (reduced)Sombor index.}
    \setlength{\tabcolsep}{0.5mm}{
	\begin{tabular}{c|cccccccccccc}\hline
	       & $DD$  & $m_{1,2}$ & $m_{1,3}$ & $m_{1,4}$ &  $m_{2,3}$   & $m_{2,4}$   &  $m_{3,3}$  &  $m_{3,4}$  &  $m_{4,4}$  & $m_{2,2}$   & $SO(G)$ & $SO_{red}(G)$\\
    \hline

    $\gamma_{41}$    &  $E_{5}$  &   $0$ &    $0$  &	  $1$   &	  $9$  &	  $3$   &	  $0$   &	  $0$  &	  $0$   &	  $n-11$   &	  $(2\sqrt{2}n+18.876)$ & $(\sqrt{2}n+17.055)$ \\ \hline

    $\gamma_{42}$    &  $E_{5}$  &   $1$ &    $0$  &	  $0$   &	  $0$  &	  $1$   &	  $3$   &	  $3$  &	  $0$   &	  $n-6$   &	  $(2\sqrt{2}n+17.465)$ & $(\sqrt{2}n+14.978)$ \\ \hline

    $\gamma_{43}$    &  $E_{5}$  &   $1$ &    $0$  &	  $0$   &	  $2$  &	  $1$   &	  $2$   &	  $3$  &	  $0$   &	  $n-7$   &	  $(2\sqrt{2}n+17.605)$ & $(\sqrt{2}n+15.208)$ \\ \hline

    $\gamma_{44}$    &  $E_{5}$  &   $1$ &    $0$  &	  $0$   &	  $4$  &	  $1$   &	  $1$   &	  $3$  &	  $0$   &	  $n-8$   &	  $(2\sqrt{2}n+17.745)$ & $(\sqrt{2}n+15.437)$ \\ \hline

    $\gamma_{45}$    &  $E_{5}$  &   $1$ &    $0$  &	  $0$   &	  $6$  &	  $1$   &	  $0$   &	  $3$  &	  $0$   &	  $n-9$  &	  $(2\sqrt{2}n+17.885)$  & $(\sqrt{2}n+15.667)$ \\ \hline

    $\gamma_{46}$   &  $E_{5}$  &   $1$ &    $0$  &	  $0$   &	  $1$  &	  $2$   &	  $3$   &	  $2$  &	  $0$   &	  $n-7$   &	  $(2\sqrt{2}n+17.714)$ & $(\sqrt{2}n+15.357)$ \\ \hline

    $\gamma_{47}$   &  $E_{5}$  &   $1$ &    $0$  &	  $0$   &	  $3$  &	  $2$   &	  $2$   &	  $2$  &	  $0$   &	  $n-8$  &	  $(2\sqrt{2}n+17.854)$ & $(\sqrt{2}n+15.587)$ \\ \hline

    $\gamma_{48}$   &  $E_{5}$  &   $1$ &    $0$  &	  $0$   &	  $5$  &	  $2$   &	  $1$   &	  $2$  &	  $0$   &	  $n-9$   &	  $(2\sqrt{2}n+17.994)$ & $(\sqrt{2}n+15.816)$ \\ \hline

    $\gamma_{49}$   &  $E_{5}$  &   $1$ &    $0$  &	  $0$   &	  $7$  &	  $2$   &	  $0$   &	  $2$  &	  $0$   &	  $n-10$   &	  $(2\sqrt{2}n+18.134)$ & $(\sqrt{2}n+16.045)$ \\ \hline

    $\gamma_{50}$   &  $E_{5}$  &   $1$ &    $0$  &	  $0$   &	  $2$  &	  $3$   &	  $3$   &	  $1$  &	  $0$   &	  $n-8$  &	  $(2\sqrt{2}n+17.964)$ & $(\sqrt{2}n+15.736)$ \\ \hline

    $\gamma_{51}$   &  $E_{5}$  &   $1$ &    $0$  &	  $0$   &	  $4$  &	  $3$   &	  $2$   &	  $1$  &	  $1$   &	  $n-9$   &	  $(2\sqrt{2}n+18.104)$ & $(\sqrt{2}n+15.965)$ \\ \hline

    $\gamma_{52}$    &  $E_{5}$  &   $1$ &    $0$  &    $0$   &	  $6$  &      $3$   &     $1$   &	  $1$  &    $0$   &     $n-10$   &
     $(2\sqrt{2}n+18.244)$  & $(\sqrt{2}n+16.195)$\\ \hline

    $\gamma_{53}$    &  $E_{5}$  &   $1$ &    $0$  &	  $0$   &	  $8$  &	  $3$   &	  $0$   &	  $1$  &	  $0$   &	$n-11$   &	  $(2\sqrt{2}n+18.384)$ & $(\sqrt{2}n+16.424)$ \\ \hline

    $\gamma_{54}$    &  $E_{5}$  &   $1$ &    $0$  &	  $0$   &	  $3$  &	  $4$   &	  $3$   &	  $0$  &	  $0$   &	  $n-9$   &	  $(2\sqrt{2}n+18.213)$ & $(\sqrt{2}n+16.114)$ \\ \hline

    $\gamma_{55}$    &  $E_{5}$  &   $1$ &    $0$  &	  $0$   &	  $5$  &	  $4$   &	  $2$   &	  $0$  &	  $0$   &	  $n-10$   &	  $(2\sqrt{2}n+18.353)$ & $(\sqrt{2}n+16.344)$ \\ \hline

    $\gamma_{56}$    &  $E_{5}$  &   $1$ &    $0$  &	  $0$   &	  $7$  &	  $4$   &	  $1$   &	  $0$  &	  $0$   &	  $n-11$   &	  $(2\sqrt{2}n+18.493)$  & $(\sqrt{2}n+16.573)$\\ \hline

    $\gamma_{57}$    &  $E_{5}$  &   $1$ &    $0$  &	  $0$   &	  $9$  &	  $4$   &	  $0$   &	  $0$  &	  $0$   &	  $n-12$   &	  $(2\sqrt{2}n+18.633)$ & $(\sqrt{2}n+16.803)$ \\ \hline

    $\gamma_{58}$    &  $E_{6}$  &   $0$ &    $1$  &	  $0$   &	  $2$  &	  $0$   &	  $6$   &	  $0$  &	  $0$   &	  $n-7$   &	  $(2\sqrt{2}n+16.030)$  & $(\sqrt{2}n+13.543)$\\ \hline

    $\gamma_{59}$    &  $E_{6}$  &   $0$ &    $1$  &	  $0$   &	  $4$  &	  $0$   &	  $5$   &	  $0$  &	  $0$   &	  $n-8$   &	  $(2\sqrt{2}n+16.170)$  & $(\sqrt{2}n+13.772)$\\ \hline

    $\gamma_{60}$    &  $E_{6}$  &   $0$ &    $1$  &	  $0$   &	  $6$  &	  $0$   &	  $4$   &	  $0$  &	  $0$   &	  $n-9$  &	  $(2\sqrt{2}n+16.310)$  & $(\sqrt{2}n+14.002)$\\ \hline

    $\gamma_{61}$   &  $E_{6}$  &   $0$ &    $1$  &	  $0$   &	  $8$  &	  $0$   &	  $3$   &	  $0$  &	  $0$   &	  $n-10$   &	  $(2\sqrt{2}n+16.450)$  & $(\sqrt{2}n+14.231)$\\ \hline

    $\gamma_{62}$   &  $E_{6}$  &   $0$ &    $1$  &	  $0$   &	  $10$  &	  $0$   &	  $2$   &	  $0$  &	  $0$   &	  $n-11$  &	  $(2\sqrt{2}n+16.590)$ & $(\sqrt{2}n+14.461)$ \\ \hline

    $\gamma_{63}$   &  $E_{6}$  &   $0$ &    $1$  &	  $0$   &	  $12$  &	  $0$   &	  $1$   &	  $0$  &	  $0$   &	  $n-12$   &	  $(2\sqrt{2}n+16.730)$ & $(\sqrt{2}n+14.690)$ \\ \hline

    $\gamma_{64}$   &  $E_{6}$  &   $0$ &    $1$  &	  $0$   &	  $14$  &	  $0$   &	  $0$   &	  $0$  &	  $0$   &	  $n-13$   &	  $(2\sqrt{2}n+16.870)$  & $(\sqrt{2}n+14.920)$\\ \hline

    $\gamma_{65}$   &  $E_{6}$  &   $1$ &    $0$  &	  $0$   &	  $1$  &	  $0$   &	  $7$   &	  $0$  &	  $0$   &	  $n-7$  &	  $(2\sqrt{2}n+15.741)$ & $(\sqrt{2}n+13.135)$ \\ \hline

    $\gamma_{66}$   &  $E_{6}$  &   $1$ &    $0$  &	  $0$   &	  $3$  &	  $0$   &	  $6$   &	  $0$  &	  $0$   &	  $n-8$   &	  $(2\sqrt{2}n+15.881)$ & $(\sqrt{2}n+13.365)$ \\ \hline

    $\gamma_{67}$    &  $E_{6}$  &   $1$ &    $0$  &    $0$   &	  $5$  &      $0$   &     $5$   &	  $0$  &    $0$   &     $n-9$   &
     $(2\sqrt{2}n+16.021)$ & $(\sqrt{2}n+13.594)$ \\ \hline

    $\gamma_{68}$    &  $E_{6}$  &   $1$ &    $0$  &	  $0$   &	  $7$  &	  $0$   &	  $4$   &	  $0$  &	  $0$   &	$n-10$   &	  $(2\sqrt{2}n+16.161)$  & $(\sqrt{2}n+13.824)$\\ \hline

    $\gamma_{69}$    &  $E_{6}$  &   $1$  &    $0$  &	  $0$   &	  $9$  &	  $0$   &	  $3$   &	  $0$  &	  $0$   &	  $n-11$   &	  $(2\sqrt{2}n+16.301)$  & $(\sqrt{2}n+14.053)$\\ \hline

    $\gamma_{70}$    &  $E_{6}$  &   $1$ &    $0$  &	  $0$   &	  $11$  &	  $0$   &	  $2$   &	  $0$  &	  $0$   &	  $n-12$   &	  $(2\sqrt{2}n+16.441)$  & $(\sqrt{2}n+14.283)$\\ \hline

    $\gamma_{71}$    &  $E_{6}$  &   $1$ &    $0$  &	  $0$   &	  $13$  &	  $0$   &	  $1$   &	  $0$  &	  $0$   &	  $n-13$   &	  $(2\sqrt{2}n+16.581)$  & $(\sqrt{2}n+14.512)$\\ \hline

    $\gamma_{72}$    &  $E_{6}$  &   $1$  &    $0$  &	  $0$   &	  $15$  &	  $0$   &	  $0$   &	  $0$  &	  $0$   &	  $n-14$   &	  $(2\sqrt{2}n+16.721)$  & $(\sqrt{2}n+14.742)$\\ \hline
	\end{tabular}}
	
	\label{table8}
\end{table}

\begin{lemma}\label{l-39}\cite{ghas2018,ggda2017}
$G\in CTG_{n}$ and $n_{1}(G)\leq 1$ if and only if $G$ belongs to one of equivalence classes given in Table \ref{table6}.
\end{lemma}

\begin{theorem}\label{t-310}
If $n\geq 6$, $G_{1}\in \gamma_{9}$, $G_{2}\in \gamma_{10}$, $G_{3}\in \gamma_{11}$, $G_{4}\in \gamma_{12}$, $G_{5}\in \gamma_{13}$, $G_{6}\in \gamma_{14}$, $G_{7}\in \gamma_{65}$  in Table \ref{table7} and Table \ref{table8}. $G\in CTG_{n}\setminus \{G_{1},G_{2},G_{3},G_{4},G_{5},G_{6},G_{7}\}$, then
$SO(G_{1})<SO(G_{2})<SO(G_{3})<SO(G_{4})<SO(G_{5})<SO(G_{6})<SO(G_{7})<SO(G)$.
\end{theorem}
\begin{proof}
By Table \ref{table7} and Table \ref{table8}, we have $SO(G_{1})<SO(G_{2})<SO(G_{3})<SO(G_{4})<SO(G_{5})<SO(G_{6})<SO(G_{7})$.

If $n_{1}(G)\leq 1$, by Table \ref{table7} and Table \ref{table8}, the conclusion holds. If $n_{1}(G)\geq 2$, by the transformations of Lemma \ref{l-21} and Lemma \ref{l-24}, we can obtain a chemical tricyclic graphs $G^{*}$ with $n_{1}(G^{*})= 1$, so we have $SO(G)>SO(G^{*})$. By Table \ref{table7} and Table \ref{table8}, $SO(G_{7})\leq SO(G^{*})$. Thus, the conclusion holds.
\end{proof}
\ \notag\

Similar to the proof of Theorem \ref{t-310}, we have
\begin{theorem}\label{t-311}
If $n\geq 6$, $G_{1}\in \gamma_{9}$, $G_{2}\in \gamma_{10}$, $G_{3}\in \gamma_{11}$, $G_{4}\in \gamma_{12}$, $G_{5}\in \gamma_{13}$, $G_{6}\in \gamma_{14}$, $G_{7}\in \gamma_{65}$  in Table \ref{table7} and Table \ref{table8}. $G\in CTG_{n}\setminus \{G_{1},G_{2},G_{3},G_{4},G_{5},G_{6},G_{7}\}$, then
$SO_{red}(G_{1})<SO_{red}(G_{2})<SO_{red}(G_{3})<SO_{red}(G_{4})<SO_{red}(G_{5})<SO_{red}(G_{6})<SO_{red}(G_{7})<SO_{red}(G)$.
\end{theorem}

\section{Applications of reduced Sombor index to octane isomers}

Deng et al.\cite{dengt2021}, considered the correlation between some physico-chemical properties of octane isomers with Sombor index.
In this section, we study the correlation between these physico-chemical properties of octane isomers with reduced Sombor index. We also compare the reduced Sombor index with some other topological indices.

\begin{figure}[ht!]
  \centering
  \scalebox{.16}[.16]{\includegraphics{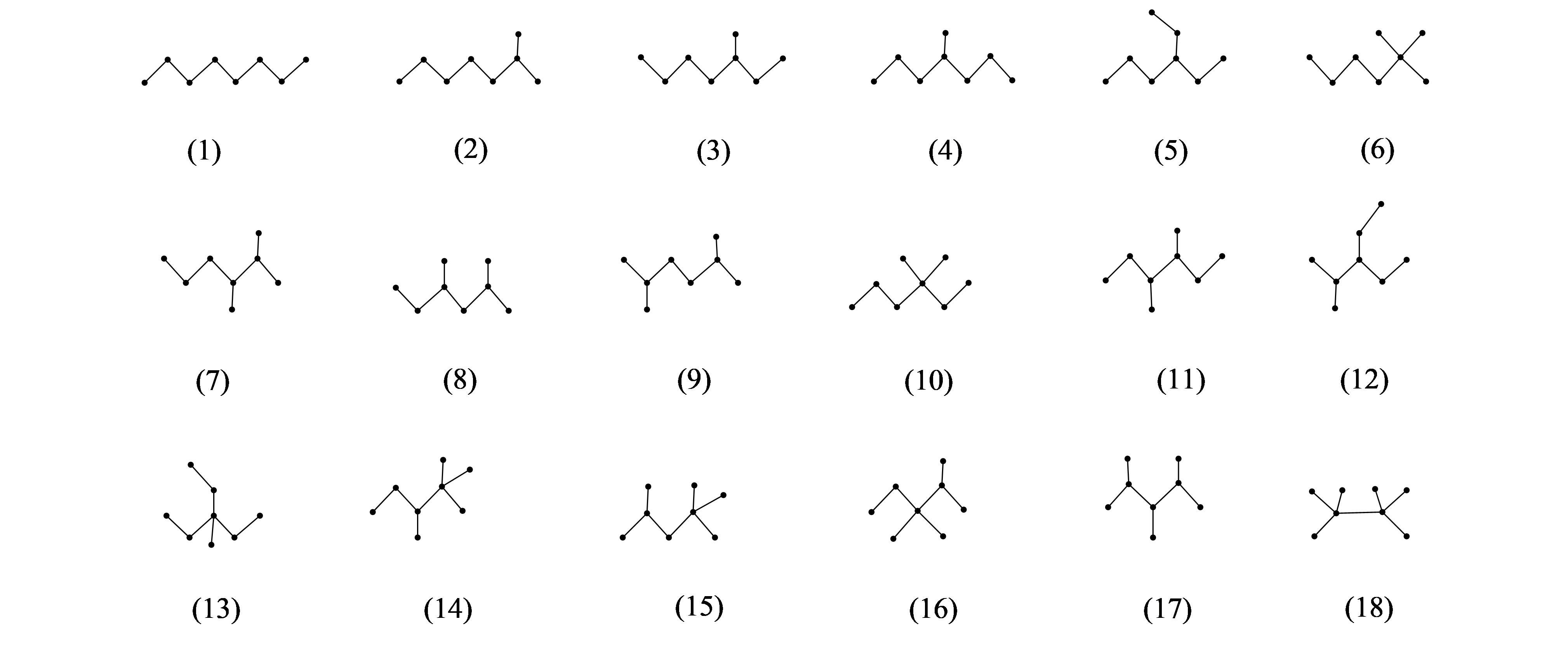}}
  \caption{Chemical graphs of octane isomers.}
 \label{fig-1}
\end{figure}

The chemical graphs of $18$ octane isomers can see in Figure \ref{fig-1}.
We can calculate the values of reduced Sombor index for the $18$ octane isomers in Figure \ref{fig-1} as $[9.0710, 11.4787, 11.30$ $05, 11.3005, 11.1224, 15.9907, 13.4787, 13.7082, 13.8663, 15.7387, 13.3005, 13.3005, 15.4868,$

\noindent$17.8416, 18.3983, 17.7678, 15.6568, 22.2426]$.
Based on the values of Acentric Factors (Entropy, SNar, HNar) of the $18$ octane isomers (see \cite{mila2021,dengt2021}) in Figure \ref{fig-1}, we can also obtain the regression models of the reduces Sombor index (similar to the results of \cite{dengt2021}).
$$AcenFac=0.4881-0.0105\times SO_{red}, R^{2}=0.9213. \eqno{(5)}$$
$$Entropy=124.5-1.317\times SO_{red}, R^{2}=0.8922. \eqno{(6)}$$
$$SNar=5.003-0.1015\times SO_{red}, R^{2}=0.9736. \eqno{(7)}$$
$$HNar=1.793-0.02654\times SO_{red}, R^{2}=0.9341. \eqno{(8)}$$

\begin{table}[!htb]
	\centering
    \caption{$R^{2}$ values between indices and Acentric Factors, Entropy, SNar, HNar.}
     \setlength{\tabcolsep}{0.7mm}{
	\begin{tabular}{c|cccccccc}\hline
	Physico-chemical property    &	   $SO_{red}$  &      $M_{1}$    &    $M_{2}$    &    $F$       &    $R$  &           $SCI$       &    $SDD$       &    $M_{N}$   \\ \hline

	Acentric Factors             &     $0.9213$    &	  $0.9468$   &	  $0.973$    &	  $0.9313$  &    $0.8176$  &      $0.8647$    &    $0.8118$    &    $0.98915$   \\ \hline

    Entropy                      &     $0.8922$    &	  $0.9107$   &	  $0.8868$   &	  $0.9077$  &    $0.8205$  &      $0.8518$    &    $0.8276$    &    $0.90746$   \\ \hline

    SNar                         &     $0.9736$    &	  $0.9974$   &	  $0.8940$   &	  $0.9453$  &    $0.9487$  &      $0.9710$    &    $0.9252$    &    $0.9477$   \\ \hline

    HNar                         &     $0.9341$    &	  $0.9774$   &	  $0.8941$   &	  $0.9453$  &    $0.9487$  &      $0.9710$    &    $0.9252$    &    $0.9115$   \\ \hline
	\end{tabular}}
	
	\label{table10}
\end{table}
The correlation $(R)$ between Acentric Factors(resp. Entropy, SNar, HNar) and reduced Sombor indices of the octane isomers is about -0.959(resp. -0.944, -0.986, -0.966). It shows a good linear relation. Therefore, the reduced Sombor index can help to predict these physico-chemical properties.
We compare the reduced Sombor index with some existing topological indices, we found that sometimes the reduced Sombor index shows better predictive power than the existing indices. It is worth noting that \cite{redz2021} (before our paper) also consider the correlation between Sombor index and Entropy of octane isomers, however, for the sake of the integrity of the article, we have not removed the results about our results of Entropy. We also considered other physico-chemical properties, such as Acentric Factors, SNar and HNar, which do not appear in the \cite{redz2021}.

\section{Concluding Remarks}

In this paper, we determine the first fourteen minimum chemical trees, the first four minimum chemical unicyclic graphs, the first three minimum chemical bicyclic graphs, the first seven minimum chemical tricyclic graphs. At last, we consider applications of reduced Sombor index to octane isomers. However, obtaining a more detailed ordering is still an open problem.

\begin{problem}\label{p6-1}
Further ordering chemical graphs by their Sombor indices.
\end{problem}

We intend to elaborate this matter in the near future.
\\

\end{document}